\documentclass{amsart}
\usepackage{latexsym,amssymb,amsmath,amsthm,latexsym}
\usepackage{amsfonts}
\theoremstyle{definition}

\newtheorem{example}{Example}
\newtheorem{corollary}{Corollary}

\newtheorem{proposition}{Proposition}

\newtheorem{remark}{Remark}

\newtheorem{identity}{Identity}
\title{On Restricted Ternary Words and Insets}
\author{Milan Janji\'c}
\date{\today}
\begin{document}
\maketitle
\begin{center}Department for Mathematics and Informatics, University of Banja
Luka,\end{center}
\begin{abstract} We investigate combinatorial properties of a  kind of
insets we defined in an earlier paper, interpreting them now in terms of
restricted ternary words. This allows us to give new combinatorial interpretations
of a number of known integer sequences, namely the coefficients of Chebyshev
polynomials of both kinds,  Fibonacci numbers, Delannoy numbers,
asymmetric Delannoy numbers, Sulanke numbers,  coordinating sequences for some
cubic lattices,   crystal ball sequences for some cubic lattices, and  others. We also obtain several new properties of said insets. In particular, we derive three generating functions when two of three variables are constant.

At the end, we state $40$ combinatorial  configurations counted by our  words.
\end{abstract}

\section{Introduction} In M. Janji\'c and B. Petkovi\'c~\cite{mb}, we defined the
notion of inset in the following way. Let $m,k,n$ be nonnegative integers. Let
$q_1,q_2,\ldots,q_n$ be positive integers. Consider the set $X$ consisting of blocks
$X_1,X_2,\ldots,X_n,Y$ such that $\vert X_i\vert=q_i,(i=1,2,\ldots,n),\vert
Y\vert=m$. Then, ${m,n\choose k,Q}$ equals the number of $(n+k)$-subsets of $X$
intersecting each block $X_i,(i=1,2,\ldots,n)$. We
call such subsets the $(n+k)$-insets of $X$.

We investigate the properties of the function ${m,n\choose k,2}$ in
which $q_i=2,(i=1,2,\ldots,n)$. To simplify our notation, we write ${m,n\choose k}$ instead of ${m,n\choose k,2}$.

Thus, we consider the three-dimensional array
\begin{equation*}A=\left\{{m,n\choose k,2}:0\leq m,0\leq n,0\leq k\leq m+n\right\}.
\end{equation*}

The following two explicit formulae for ${m,n\choose k}$ are proved
in~\cite[Equations(8) and (11)]{mb}.
\begin{equation}\label{e1}{m,n\choose k}=\sum_{i=0}^n(-1)^i{n\choose
i}{m+2n-2i\choose n+k}.
\end{equation}
\begin{equation}\label{e2}{m,n\choose k}=2^{n-k}\sum_{i=0}^m2^i{m\choose i}{n\choose
k-i}.
\end{equation}
In particular, we have ${0,0\choose 0}=1$.
\begin{remark} We note that ${m,n\choose k}$ are even if $n>k$. In the case $n=k$,
Equation (\ref{e2}) becomes
\begin{equation*}{m,n\choose n}=\sum_{i=0}^m2^i{m\choose i}{n\choose
i},\end{equation*}
which is  the well-known formula for the Delannoy numbers $D(m,n)$.
\end{remark}
From~\cite[Eq. (11)]{mb}, we obtain the following identity:
\begin{equation}\label{eq1}{m,n\choose k}={m,n-1\choose k}+{m+1,n-1\choose k},
\end{equation}
Using this identity, we easily  express the Fibonacci numbers in terms of insets.
\begin{corollary}\label{fbb} For $m\geq 0$, we have
\begin{equation*}F_{m+3}=\sum_{i=0}^{\lfloor\frac{m+1}{2}\rfloor}{m-i,1\choose i}.
\end{equation*}
\end{corollary}
\begin{proof}
From (\ref{eq1}), we obtain
\begin{equation*}\sum_{i=0}^{\lfloor\frac{m+1}{2}\rfloor}{m-i,1\choose i}=
\sum_{i=0}^{\lfloor\frac{m+1}{2}\rfloor}\left[{m-i\choose i}+{m-i-1\choose i}\right],
\end{equation*}
and the proof follows from the  well-known  expression of the Fibonacci numbers
 in terms of the binomial coefficients.
\end{proof}
From~\cite[Eq. (13)]{mb}, we obtain the following identity:
\begin{equation}\label{eq2} {m,n\choose k}=2{m,n-1\choose k}+{m,n-1\choose k-1}.
\end{equation}

We proceed by deriving some new properties of the function ${m,n\choose k}$.
\begin{proposition} For $0\leq p\leq m;0<n$, the following formula holds:
\begin{equation}\label{ee}
{m+1,n-1\choose k}=\sum_{i=0}^{p}(-1)^i{p\choose i}{m-p+1,n+p-1-i\choose k}.
\end{equation}
\end{proposition}
\begin{proof}The formula is obviously true for $p=0$.
We next write (\ref{eq1}) in the following form:
\begin{equation}\label{e1e}
{m+1,n-1\choose k}={m,n\choose k}-{m,n-1\choose k}.
\end{equation}
This means that (\ref{ee}) holds for $p=1$.

Applying $(\ref{e1e})$ on the right-hand side of the same formula yields
\begin{equation*}
{m+1,n-1\choose k}={m-1,n+1\choose k}-2{m-1,n\choose k}+{m-1,n-1\choose k}.
\end{equation*}
This means that (\ref{e1e}) holds for $p=2$. Repeating the same procedure, we  obtain that (\ref{ee}) is true for each $p\leq m$.
\end{proof}

 \begin{proposition}\label{nn}
The following formula holds: \begin{equation}\label{rlr}{m+1,n\choose
k+1}=2^{n-k-1}{n\choose k+1}+\sum_{i=0}^m{i,n\choose k}.\end{equation} In particular,  for $n\leq k\leq m+n$, we have \begin{equation}\label{r33}{m+1,n\choose k+1}=\sum_{i=0}^m{i,n\choose k}.\end{equation}
 \end{proposition}
 \begin{proof} From~\cite[Proposition 10]{mb}, it follows that
\begin{gather*} {m+1,n\choose k+1}={m,n\choose k}+{m,n\choose k+1}=
{m,n\choose k}+{m-1,n\choose k}+{m-1,n\choose k+1}\\ ={m,n\choose
k}+\cdots+{0,n\choose k}+{0,n\choose k+1}.
 \end{gather*} Taking $m=0$ in (\ref{e2}),
we obtain ${0,n\choose k+1}=2^{n-k-1}{n\choose k+1}$, which proves (\ref{rlr}). The
statement (\ref{r33}) is obvious.
\end{proof}
\begin{remark} The equation
(\ref{r33}) is an analog of the horizontal recurrence for the binomial coefficients.
\end{remark}
\begin{proposition} Let $p$ be a positive integer such that $1\leq
p\leq\min\{n,k\}$. Then, \begin{equation}\label{ip}{m,n-p\choose k-p}-{m,n\choose
k}=2\cdot\sum_{i=1}^p{m,n-i\choose k-i+1}.\end{equation}
\end{proposition}
\begin{proof} From Equation (\ref{eq2}), we obtain the following sequence of
equalities:
\begin{gather*} {m,n\choose k}-{m,n-1\choose k-1}=2{m,n-1\choose k},\\
{m,n-1\choose k-1}-{m,n-2\choose k-2}=2{m,n-2\choose k-1},\\ \vdots\\
{m,n-p+1\choose
k-p+1}-{m,n-p\choose k-p}=2{m,n-p\choose k-p+1}.
\end{gather*}
The sum of the expressions on the left-hand side and the sum of the expressions on the right-hand side of the
preceding equations produces Equation (\ref{ip}).
 \end{proof}
\begin{corollary}The numbers
${m,n-p\choose k-p}$ and ${m,n\choose
k}$ are of the same parity.
\end{corollary}
\begin{remark} Note that for $n=0$, the array $A$ is the standard Pascal triangle.
\end{remark}

 \begin{proposition}\label{pp1} For a fixed positive $n$, the array
$\{{m,n\choose k}:(m,k=0,1,\ldots)\}$ is a Pascal trapeze. The elements of the
leftmost diagonal equal  $2^n$, the rightmost diagonal consists of ones, and the first row consists of $2^{n-k}{n\choose k},(k=0,1,\ldots,m+n)$.
\end{proposition}
\begin{proof} In the case $k=0$, we have ${m,n\choose 0}$, which is the number of
insets of $n$ elements. In such an inset, we must choose exactly one element from
each main block. We have $2^n$ such insets. We conclude that the leftmost
diagonal of the array $A$ is a constant sequence consisting of $2^n$. Furthermore, if
$n+k>m+2n$, then ${m,n\choose k}=0$, since there is no inset having more than
$m+2n$ elements. If $k+n=m+2n$, that is, if $k=m+n$, then ${m,n\choose k}=1$.
Hence, the rightmost diagonal consists of ones. Since the first row is obtained for $m=0$, Proposition  \ref{e2} yields  ${0,n\choose k}=2^{n-k}{n\choose k}, (k=0,1,\ldots, m+n)$.
Finally, from~\cite[Proposition 10]{mb}, the following (Pascal) recurrence holds:
\begin{equation}\label{fm1k}{m,n\choose k}={m-1,n\choose k-1}+{m-1,n\choose k}. \end{equation}
 \end{proof}
\section{A combinatorial interpretation of
${m,n\choose k}$}

In this section, we relate the array $A$ to  restricted ternary
words.  We firstly consider the insets counted by  ${0,n\choose k}$.
Note that ${0,n\choose k}=0$, if $k>n$. Also, we have ${0,n\choose n}=1$.
Let $X$ be the set
\begin{proposition}\label{pr1} The number
${0,n\choose k}$ equals the number of ternary words of length $n$ having $k$ letters equal to $2$.
\end{proposition}
 \begin{proof}
  Assume that $k\leq n$.
  From each $(n+k)$-inset of $X$, we form a ternary word of length $n$ in the following way. Put $2$ in the place of every main block from which both elements are in $X$.
Then, insert $1$ in the place  of every main block from which the only first  element is in $X$, and fill all other places with zeros. It is clear that this
correspondence is bijective.
\end{proof}

 We now consider the case  $m>0$.
\begin{proposition}\label{pr2} For $k\leq m+n$, the number  ${m,n\choose k}$ is the number of ternary words of length $m+n$ having $k$ letters equal to $2$ and none of the first $m$ letters  is equal to $0$.
  \end{proposition}
\begin{proof} The following formula  is proved in~\cite[Proposition 11]{mb}:
\begin{equation*}{m,n\choose k}=\sum_{i=0}^m{m\choose i}{0,n\choose k-i}.
\end{equation*}
Take $i\in\{0,1,\ldots,m\}$. According to Proposition \ref{pr1}, ${0,n\choose
k-i}$ is the number of ternary words of length $n$ having $k-i$ letters equal to $2$.
On the other hand, ${m\choose i}$ is the number of words of length $m$ over $\{1,2\}$
having $i$ letters equal to $2$. Hence, ${m\choose i}{0,n\choose k-i}$ is the number of
ternary words of length $m+n$, having $k$ letters equal to $2$, and beginning with a subword of length $m$ over $\{1,2\}$. Summing over all $i$, we obtain the assertion.
\end{proof}

\begin{example}
 \begin{enumerate}
\item For $m=1,n=3$, and $k=2$, we have ${1,3\choose 2}=18$. The corresponding words counted by this number are
\begin{align*}1022,1122,1202,1212,1220,1221,2200,2211,2210,\\
2201,2020,2121,2021,2120,2002,2112,2012,2102.
\end{align*}
\item For $m=1,n=3$, and $k=3$, we have ${1,3\choose 3}=7$. The corresponding
words are
\begin{equation*}1222,2122,2022,2212,2202,2221,2220.
\end{equation*}
 \item For $m=2,n=3$, and $k=4$, we have ${2,3\choose 4}=8$. The corresponding
 numbers are
\begin{equation*}12222,21222,22122,22212,22221,22220,22202,22022.
\end{equation*}
\end{enumerate}
\end{example}
We derive a  property of the function ${m,n\choose k}$ by introducing  a new parameter $p$.
\begin{identity}\label{pr11} If $0\leq p\leq n$, then
\begin{equation}\label{eq3}{m,n\choose
k}=\sum_{i=0}^p{p\choose i}{m+i,n-p\choose k}.
\end{equation}
\end{identity}
\begin{proof}
The result may be proved  by induction on $p$, using Equation (\ref{eq1}).
We add a combinatorial proof.

For $0\leq i\leq p\leq n$, we put $p-i$ zeros  in some of the last $p$ places in a word. This may be done in ${p\choose i}$ ways.
Now, we have to choose the remaining  $m+n+i-p$ letters. In other words, we have to count the ternary words of length $m+n-p+i$ satisfying the following conditions:
\begin{enumerate}
\item A word must have $k$ letters equal to $2$.
\item Zeros may stand in the positions $m+1,m+2,\ldots,m+n-p$.
\end{enumerate}
It is clear that we have ${m+i,n-p\choose k}$ such words. Summing over $i$ from $0$ to $p$, we obtain the equation.
  \end{proof}
In particular, taking $p=n$, we obtain
\begin{corollary} The following formula holds: \begin{equation}\label{ee1}{m,n\choose
k}=\sum_{i=0}^n{n\choose i}{m+i\choose k}.
\end{equation}
 \end{corollary}
We derive combinatorially one more formula in which ${m,n\choose k}$ is expressed  as a convolution of  binomial coefficients.
\begin{proposition}
The following formula holds:
\begin{equation*}{m,n\choose k}=\sum_{i=0}^n\sum_{j=0}^i{n\choose i}{i\choose j}{m\choose k-i+j}.
\end{equation*}
\end{proposition}
\begin{proof}
We firstly designate positions of $0$ in a word.

xxxxxxxxxxxxxx
 Since zeros may be only some of the last $n$ letters in a word, we choose $n-i$ zeros for $i=0,1,\ldots n$. This may be done in ${n\choose i}$ ways.

 For the remaining $i$ of the last $n$ letters in a word, we firstly choose $i-j$ letters equal to $2$, which may be done in ${i\choose i-j }$ ways. The remaining $j$ letters are equal to $1$.

 Furthermore, for fixed $i$ and $j$, among the first $m$ letters in a word, we have to choose $k-i+j$ letters equal to $2$, and the remaining letters must be equal to $1$. This may be done in ${m\choose k-i+j}$ ways. Summing over all $i$ and $j$, we obtain the result.
\end{proof}
\begin{corollary}
    B. Braun, W. K. Hough~\cite{wkh} showed that  terms  of the array $\{{2,n\choose k}: n=0,1,\ldots;k=0,1,\ldots,n+2\}$ count cells in the  cellular complex $X^m_n$ defined in \cite[Definition 3.4.]{wkh}.
 More precisely, let  $C^d_n$ denote the number of $d$-dimensional cells in $X^2_n$.
It is shown in \cite[Proposition 4.6.]{wkh} that the following formula holds:
\begin{equation*}C^d_n={2,n-d+2\choose 3d-2n}.\end{equation*}
\end{corollary}
\begin{corollary}
In Hetyei's paper \cite[Definition 2.1]{ht}, the numbers on the form ${m,n\choose m},(m=0,1,\ldots;n=0,1,\ldots)$ are called the asymmetric Delannoy numbers.

We note that rows of this array are diagonals of A049600 in~Sloane\cite{slo}.
\end{corollary}
 In our previous paper~\cite{mj}, we defined  a class of polynomials which generalize the Tchebychev polynomials. We state a particular result concerning  our present investigations.

 For a fixed nonnegative integer $m$ and nonnegative integers
$n,k$ such that $0\leq k\leq n,0\leq m\leq\frac{n+k}{2}$, we defined polynomials
$P_{m,n}(x)=\sum_{k=0}^nc_m(n,k)x^k$ such that
\begin{equation}\label{mcon}c_m(n,k)=(-1)^{\frac{n-k}{2}}{m,\frac{n+k}{2}-m\choose
\frac{n-k}{2}}, \end{equation} if $n$ and $k$ are of the same parity and
$c_m(n,k)=0$ otherwise. Using Proposition 2, we obtain the following combinatorial interpretation of the coefficients of the polynomials $P_{m,n}(x)$.

\begin{example}\label{co1}Assume that $n$ and $k$ are of the same parity.  Let
$c_m(n,k)$ denote  the coefficient of $x^k$ of the polynomial $P_{m,n}(x)$.  Then
$\vert c_m(n,k)\vert$  equals the number of ternary words of length $\frac{n+k}{2}$
 having $\frac{n-k}{2}$ letters equal to $2$ and no zeros among the first $m$
 letters in a word.
\end{example}
 \begin{remark}In this way we give a new combinatorial
interpretation of the coefficients of polynomials which appear in~Sloane\cite{slo}
A136388, A136389, A136390, A136397, and
A136398.
 \end{remark}

\section{Three generating functions}
I this part, we derive three generating functions assuming that two of parameters $m,n,k$ are constant. Firstly, we prove an identity.
\begin{proposition}For $m+k\geq n$, we have
\begin{equation}\label{p14}
{m+k-n,n\choose k}=\sum_{i=0}^m{n\choose m-i}{k+i\choose k}.
\end{equation}
\end{proposition}
\begin{proof} We have to count ternary words of length $m+k$, having $k$ letters
equal to $2$, and zeros may  appear  only  among last $n$ letters in a word.
 It is clear that the maximal number of zeros is $m$. Hence, for $0\leq i \leq m$, we choose $m-i$ zeros, which
may be done in ${n\choose m-i}$ ways. On the remaining $k+i$
places, we choose $k$ twos, which may be done in ${k+i\choose k}$ ways. The remaining letters are equal to $1$. Summing over all $i$, we obtain the formula.
\end{proof}
We next prove three formulas concerning  ordinary generating functions for  insets,
when two of parameters $m,n,k$ are constant.
\begin{proposition}
If  $m+k\geq n$, then the following expansion holds.
\begin{equation}\label{mkn}\frac{(1+x)^n}{(1-x)^{k+1}}=\sum_{m=m_0}^\infty{m+k-n,n\choose
k}x^m,
\end{equation}
where $m_0=\max\{0,n-k\}$.
\end{proposition}
\begin{proof}
In the case $k\geq n$, the formula is a particular case of \cite[Theorem 1]{mm}.
So, we consider the case $n>k$ and $m\geq n-k$. It is well-known fact that
\begin{equation*}
\frac{1}{(1-x)^{k+1}}=\sum_{i=0}^\infty{k+i\choose i}x^i.
\end{equation*}
Multiplying by $(1+x)^n$, we obtain
\begin{equation}\label{exx}
\frac{(1+x)^n}{(1-x)^{k+1}}=\sum_{i=0}^\infty\sum_{j=0}^n{k+i\choose k}{n\choose
j}x^{i+j}.
\end{equation}
We calculate the coefficient $a_{m+k-n}$ of $x^{m+k-n}$ on the right-hand side of
this equation. It is obtained for $i=0,1,2,\ldots,m+k-n;j=m+k-n-i$
and is equal to
\begin{equation*}a_{m+k-n}=\sum_{i=0}^p{n\choose p-i}{k+i\choose i},
\end{equation*}
where $p=m+k-n$. Applying (\ref{p14}), we obtain that $a_{m+k-n}={m+k-n,n\choose
k}$.
\end{proof}
The preceding generating function is with respect to the parameter $m$. The following is with respect to $n$.
\begin{proposition} For $n+k\geq m$, we have
\begin{equation*}
\frac{(1-x)^m}{(1-2x)^{k+1}}=\sum_{n=0}^\infty{m,n+k-m\choose k}x^n.
\end{equation*}
\end{proposition}
\begin{proof} Similarly to the proof of the preceding proposition, we reduce the proof to the following identity:
\begin{equation*}
{m,n+k-m\choose k}=\sum_{i=0}^n(-1)^i\cdot 2^{n-i}\cdot{m\choose i}\cdot{k+n-i\choose k},
\end{equation*}
which  easily follows from (\ref{ee}).
\end{proof}
Finally, we derive a generating function with respect to parameter $k$.
\begin{proposition} We have
\begin{equation*}
\frac{(2-x)^n}{(1-x)^{m+n+1}}=\sum_{k=0}^\infty{m+k,n\choose k}x^k.
\end{equation*}
\end{proposition}
\begin{proof}
We have
\begin{equation*}
\left(\frac{2-x}{1-x}\right)^n=\left(1+\frac{1}{1-x}\right)^n=\sum_{i=0}^n{n\choose i}\frac{1}{(1-x)^i}.
\end{equation*}
It follows that
\begin{equation*}
\frac{(2-x)^n}{(1-x)^{m+n+1}}=\sum_{i=0}^n{n\choose i}\frac{1}{(1-x)^{i+m+1}}=
\sum_{k=0}^\infty\sum_{i=0}^n{n\choose i}{k+i+m\choose k}x^k.
\end{equation*}
On the other hand, using  (\ref{ee1})  we obtain
\begin{equation*}
\sum_{i=0}^n{n\choose i}{k+i+m\choose k}={m+k,n\choose k}.
\end{equation*}
\end{proof}

We finish with two sets of examples. The first concerns two dimensional arrays
 in Slone~\cite{slo} generated by our function. The second  concerns sequences in Slone~\cite{slo}. Some examples are new, and some are only a new  combinatorial interpretations of our earlier results.
 \section{Examples I}
\begin{example}
It follows  from Proposition \ref{pp1} that the array $A=({m,1\choose k}),(m=0,1,\ldots;k=0,1,\ldots,m+1)$ is $(2,1)$ Pascal triangle A029653.
It is  the mirror of Lucas triangle, which is A029635.

We thus obtain the following combinatorial interpretation of entries of Lucas triangle:   The $(k,m)$ entry of Lucas  triangle equals the number of ternary words of length $m+1$ having $k$ letters equal to  $2$ and $0$ may only be the last letter of a word.
\end{example}
\begin{example}
From (\ref{e2}) follows
${0,n\choose k}=2^{n-k}{n\choose k}$.  A038207

 Hence,  the number $2^{n-k}{n\choose k}$ equals the number of ternary words of length $n$ having $k$ letters equal to $2$.

For instance, for $n=3,k=2$, we have the following $6$ words:
\begin{equation*}221,212,122,220,202,022.
\end{equation*}
\end{example}
\begin{example}
Taking $m=0$ in (\ref{mcon}), we obtain $P_{0,n}(x)=U_n(x)$, where
$U_n(x)$ is the Tchebychev polynomial of the second kind.
If $u_{n,k}$ is the coefficients of $x^k$ of
$U_n(x)$, then
$\vert u_{n,k}\vert$ equals the number of ternary words of length $\frac{n+k}{2}$
 having $\frac{n-k}{2}$ letters equal to $2$. A038207

 For instance, we have $\vert u_{4,2}\vert(=12)$, and the corresponding words are:
\begin{equation*}200,201,210,211,020,120,021,121,002,102,012,112.
 \end{equation*}
\end{example}
\begin{example}
  Taking $m=1$ in (\ref{mcon}), we obtain $P_{1,n}(x)=T_n(x)$, where $T_n(x)$ is the Tchebychev polynomial of the first kind.
Hence, if $t_{n,k}$ is the coefficients of $x^k$  of $T_n(x)$, then $\vert t_{n,k}\vert$ equals the number of ternary words of length $\frac{n+k}{2}$
 having $\frac{n-k}{2}$ letters equal to $2$ and $0$ is not the first letter of a word. seqnum{A200139}

 In particular, for $\vert t_{4,2}\vert(=8)$, we obtain the following words:
 \begin{equation*}112,121,211,120,102,210,201,200.
 \end{equation*}
  \end{example}

\begin{example}
 The array $\{{3,n\choose k}: n=0,1,\ldots;k=0,1,\ldots,n+3\}$ is in a
 way connected with array
  seqnum{A289921}.

   Namely, arrays have a number terms which are the
  same, but not all.
\end{example}

\begin{example}
In~\cite[Proposition 45]{mb}, we proved that ${m+1,n-1\choose m}$ equals the number of weak compositions of $m+n$ having $n-1$ parts equal to zero. Hence, this number equals the number of ternary words of length $m+n$ having $m$ letters equal to $2$
and the initial subword of length $m+1$ contains no zero.

In particular, we have $\{{3,n\choose 2}=2^{n-3}(n^2+11n+24),(n\geq 0)$.
seqnum{A058396}		

 The number counts weak compositions of $n+1$ having exactly $2$ parts equal to $0$. For instance, for $n=1$, we have the following $9$ weak composition of $2$ having
two zeros:
\begin{equation*}200,020,002,1100,1010,1001,0110,0101,0011.
\end{equation*}
Ternary words of length $4$ having $2$ letters equal to $2$, and the initial subword of length $3$ contains no zero are:
\begin{equation*}2211,2121,2112,1212,1221,1122,1220,2120,2210.
\end{equation*}
\end{example}

  In~\cite[Proposition 27]{mb}, the following explicit formula for
the Delannoy numbers $D(m,n)$ is derived: \begin{equation*}D(m,n)={m,n\choose n}.\end{equation*}

We have the following  combinatorial interpretation of the Dalannoy numbers.
seqnum{A008288}
\begin{example}
 The
Delannoy number $D(m,n)$ equals the number of ternary words of length $m+n$ having $n$
letters equal to $2$ and the initial subword of length $m$ contains no zero.
\end{example}
The  following example also concerns the Delannoy numbers.
\begin{example}
In~\cite[Proposition 31,1.]{mb}, we proved that the number of solutions of the Diophantine inequality
\begin{equation*}\vert x_1\vert+\vert x_2\vert+\cdots+\vert x_n\vert\leq m
\end{equation*}
equals ${m,n\choose n}(=\sum_{i=0}^n{n\choose i}{m-i+n-1\choose n-1})$. We denote
this number by $G(m),(m=0,1,\ldots)$.
\end{example}
\begin{remark}
It follows from~Conway and Sloane \cite[Eq. (3.3)]{cs} that $G(m)$ are
the christal ball numbers for the cubic lattice $\mathbb Z^n$.
\end{remark}

\begin{example}In~\cite[Proposition 31,2.]{mb}, we proved that the number of solutions of the Diophantine equality
\begin{equation*}\vert x_1\vert+\vert x_2\vert+\cdots+\vert x_n\vert=m
\end{equation*}
equals ${m-1,n\choose n-1}(=\sum_{i=0}^n{n\choose i}{m-i+n-1\choose n-1})$. We denote this number by $S(m),(m=0,1,\ldots)$.
\end{example}
Hence, $S(m)$ equals the number of ternary words of length $m+n-1$ having $n-1$ letters equal to $2$ and the initial subword of length $n$ contains no zero.

Some sequence in Sloane~\cite{slo} concerning this case are:
A005899, A008412, A008413, A008414, A008415.
\begin{remark}
It follows from~Conway and Sloane \cite[Eq. (3.2)]{cs} that $S(m),(m=0,1,\ldots)$. is the coordinating sequence for the  cubic lattice $\mathbb Z^n$.
\end{remark}
\begin{example} We now consider the assymetric Dellanoy numbers
${m,n\choose m},(n,k\geq 0)$. A049600

The assymetric Delannoy  number ${m,n\choose m}$ equals the number of ternary words of length $m+n$ havin $m$ letters equal to $2$ and
the initial subword of length $m$ contains no zero.

For $m=n=2$, we have ${2,2\choose 2}=13$, and   ternary words are:
six permutation of $1122$, three permutation of $122$ ending by $0$, three permutations of $122$ with $0$ at the next to the last place, and $2200$.
\end{example}

\begin{example}
From ~\cite[Proposition 45]{mb} follows that
  the number of composition of $m$ in which   $k$ parts are equal to zero, which equals  ${m+1,k-1\choose k}$ is the number of ternary words of length $m+k$ having $k$  letters equal to $2$ and
  the initial subword of length $m+1$ contains no zero.

  The rows of array $\{{m+1,k-1\choose k}:m=0,1,\ldots;k=1,2,\ldots\}$
  are figurate numbers based on the $k$-dimensional regular convex polytope.

 Some sequences related to this case are: A005900, A014820, A069038, A069039, A099193.

 \end{example}
\begin{example}
We next consider  the Sulanke numbers, which we denote by $s(n,k)$ . It is proved
in~\cite[Proposition
29]{mb} that $s(n,k)={\frac{n+k}{2},\frac{n+k}{2}\choose k}$, if $n+k$ is even.
      In this case, the number $s_{n,k}$  equals the number of ternary words of length $n+k$ having
$k$ twos and the initial subword of length $\frac{n+k}{2}$ contains no zero.
A064861
\end{example}
\begin{example}
If $n+k$ is odd, the
$s(n,k)={\frac{n+k-1}{2},\frac{n+k+1}{2}\choose k}$.
Now, the number $s_{n,k}$ is the number of ternary words of length $n+k$ having $k$ twos  and the initial subword of length $\frac{n+k-1}{2}$ contains no zero.
A064861
 \end{example}

\section{Examples 2}
\begin{example}
Values of ${m,1\choose 1}$ are odd numbers $2m+1,(m=0,1,\ldots)$. seqnum{A005408}

Hence, $2m+1$ equals the number of ternary words of length $m+1$ having one letter equal to $2$ and and the initial subword of length $m$ contains no zero.

For instance, for $m=2$, five ternary words of length $3$ are $112,121,211,120,210$.
\end{example}
\begin{example}
Values of ${m,1\choose 2}$ are squares $m^2,(m=0,1,\ldots)$. A000290

Hence, $m^2$ equals the number of ternary words of length $m+1$ having two letters equal to $2$ and the initial subword of length $m$ contains no zero.

For instance, for $m=3$, nine ternary words of length $4$ are \begin{equation*}1122,1212,1221,2121,2211,2112,2210,2120,1220.\end{equation*}
\end{example}

\begin{example}
For $m\geq 2$ velues of ${m,1\choose 3}(=\frac{(m-1)m(2m-1)}{6},(m=1,2,\ldots))$ are square pyramidal numbers
A000330

Hence, the square pyramidal number $\frac{(m-1)m(2m-1)}{6}$ equals the number of ternary words of length $m+1$ having three letters equal to $2$ and the initial subword of length $m$ contains no zero.

For instance, for $m=4$, $14$ ternary words of length $5$ are:
$10$ permutations of $11222$ and $4$ permutation of $1222$ ending by $0$.
\end{example}
\begin{example}
For $m\geq 2$ values of ${m,1\choose 4}$ are four dimensional pyramidal numbers:  $\frac{(m+1)^2((m+1)^2-1)}{12}$.  A002415

Hence, $\frac{(m+1)^2((m+1)^2-1)}{12}$ equals the number of ternary words of length $m+2$ having four letters equal to $2$ and the initial subword of length $m$ contains no zero.

For instance, for $m=4$, $20$ ternary words of length $6$ are:
$15$ permutations of $112222$ and $5$ permutation of $12222$ ending by $0$.
\end{example}
\begin{example}
  We have
${m,2\choose 2}=m^2+(m+1)^2$. It is seqnum{A001844} in~\cite{slo}. Hence, the sum
$m^2+(m+1)^2$ equals the number of ternary words of length $m+2$ having $2$
letters equal to $2$ and the initial subword of length $m$ contains no zero.

 For
$m=2$ this number equals $13$. The corresponding ternary words are
\begin{equation*}2211,2121,1221,1212,1122,2112,2210,2120,1220,2201,2120,1220,2200.\end{equation*}
\end{example}		

\begin{example}
 We have  ${m,2\choose 3}=\frac{m(2m^2+1)}{3}$, which are the octahedral
numbers. A005900

 Hence,  the octahedral  number
$\frac{m(2m^2+1)}{3}$ equals the number of ternary words of length $m+2$ having
three letters equal to $2$ and the initial subword of length $m$ contains no zero.

In particular, we have
${3,2\choose 3}=19$ and the corresponding ternary words have length $5$ with three
letters equal to $2$ and zero may stand only at the last two position. The word
consists of $10$ permutation of $22211$, $8$ permutation of $2221$ ending with
either $1$ or $0$, and $22200$.
\end{example}
\begin{example}
 The sequence  ${m,3\choose 3},(m=0,1,\ldots)$ consists
    of the centered octahedral numbers $\frac{(2m+1)(2m^2+2m+3)}{3}$.
    A001845

  Hence, the centered octahedral  number
$\frac{(2m+1)(2m^2+2m+3)}{3}$ equals the number of ternary words of length $m+3$ having
three letters equal to $2$ and the initial subword of length $m$ contains no zero.

\end{example}
\begin{example}
Sequence  ${m,2\choose 4},(m=0,1,2,\ldots)$ consists of
$4$-dimensional analog of centered polygonal numbers $\frac{m(m-1)(m^2-m+1)}{6}$.
A006325		
\end{example}
 \begin{example}
  We have   ${1,n\choose 2}=n(n+3)\cdot 2^{n-3},(n=0,1,\ldots)$.  A001793

The number ${1,n\choose 2}$ equals the number of Dyck paths of semilength $n+2$ having pyramid weight equal to  $n+1$ equals. It is equal to  the number of ternary words of length $n+1$ having two letter   equal to $2$ and not beginning by $0$.
\end{example}
\begin{example}
    In~\cite[Proposition 50]{mb}, we proved that for $n>2$, the possible bishop moves on $n \times n$ chessboard equals ${1,n\choose n-2}(=\frac{2n(2n-1)(n-1)}{3})$. A002492

    Hence,  for $n>2$ the possible bishop moves on $n \times n$ chessboard equals
     the number of ternary words of length $n+1$ having $n-2$ letters equal to $2$, and    not beginning by $0$.
 \end{example}
 \begin{example}
  For $m\geq 2$, the number ${m,2\choose 5},(=\frac{m(m^4-1)}{30})$
is  the convolution of  nonzero squares with themselves, that is,
${m,2\choose 5}=\sum_{i=2}^mi^2(m-i)^2$, and  is  equal to the total area of all
square regions from an $m\times m$ grid.  A033455
\end{example}

\begin{example}
The central Delannoy number $D(n,n)$ is the number of ternary words of length $2n$
having $n$ letters equal to $2$ and $0$ can not appear in the initial subword of
length $n$. A001850
 \end{example}
\begin{example}
We derive a result concerning the Catalan numbers. A000108

 It is proved in~\cite[Proposition 25]{mb} that for  the $k$th
Catalan number $C_k$ holds
 \begin{equation*}C_k=\frac{1}{3k+2}{2k,1\choose
k}.\end{equation*}
Hence, $(3k+2)\cdot C_k$ equals the number of ternary words of
length $2k+1$ having $k$ letters equal to $2$ and $0$ may appear only as the last
letter in a word. A051960

In particular, for $k=2$ we have $(3k+2)C_2=16$. These $16$ word of length $5$ are
$10$ permutation of $11122$, and $6$ permutation of $1122$ ending by zero.
  \end{example}

\begin{example}
In this example, we express the Fibonacci numbers in terms of ternary words. A000045

 Let $m$ be a fixed positive integer. From (\ref{fbb}) follows that  $F_{m+3}$ equals the number of ternary words of length $m-i+1,(0\leq i\leq m)$  having $i$ letters equal to $2$ and $0$ may only be the last letter in a word.
In particular, $F_6=8$
and the corresponding ternary words are
\begin{equation*}1111,1110,112,121,211,210,120,22.
\end{equation*}
\end{example}
  \begin{example} The number of triangles in the Turan graph $T(m, m-2)$, for $m>3$,
 equals the number of ternary words of length $m+3$
 having $m$ letters equal to $2$ and $0$ can not appear in the initial subword of
length $m+1$.

 The number is $\{{m+1,2\choose m}$.  A000297.
\end{example}
\begin{example}
The number of points on surface of octahedron equals the number of ternary words of length $m+2$ having two letters equal to $2$ and $0$ can not appear in the initial subword of length $m-1$.

For $m>1$, the number is ${m-1,3\choose 2}(=4m^2+2)$.
A005899
\end{example}

\begin{example}
 The number of maximal and maximum cliques in the $n$-cube-connected cycles graph
 equals the number of ternary words of length $2n$ having one letter equal to $2$ and $0$ can not appear in the initial subword of
length $n$.

 The number is  $3n\cdot 2^{n-1}(={n,n\choose 1})$.
   A167667
\end{example}
 \begin{example}
 The number of peaks in all Schroeder paths  from $(0,0)$ to $(2n,0)$ equals
the number    of ternary words of length $2m+1$ having $m+1$ letters equal
   to $2$ and $0$ can not appear in the initial subword of
length $m$.
    The number is $\sum_{i=1}^{m+1}{m+1\choose i}{m+i\choose
     i-1}(={m,m+1\choose m+1})$.
    A002002
\end{example}
\begin{example}
 The number of order-preserving partial self maps of $\{1,\ldots,m\}$.
    equals the number of ternary words of length $2m+1$ having $m$ letters equal to $2$ and $0$ can not appear in the initial subword of
length $m$.

  The number is $\sum_{i=0}^{m+1}{m+1\choose i}{m+i\choose i}(={m,m+1\choose m})$.
    A002003
\end{example}
\begin{example}
  The number of Dyck paths with semilength $m+4$, and an odd number of peaks, and the central peak has height $m-2$
   equals the number of ternary words of length $m+2$ having $m$ letters equal to
   $2$ and  a word can not begin with $0$.

The number is $2(m+1)^2(={1,m+1\choose m})$.
   A001105
 \end{example}
\begin{example}
     The sum  of the  first $m+1$ even squares
   equals the number of ternary words of length $m+3$ having $m$ letters equal to
   $2$ and $0$ can not appear in the initial subword of
length $m$.

 The number is $\frac{2(m+1)(m+2)(2m+3)}{3}(={m,m+2\choose m})$.
   A002492
\end{example}
\begin{example}
 The variance of time for a random walk starting at $0$ to reach one of the
 boundaries at $+m$ or $-m$ for the first time    equals the number of ternary
 words of length $m+4$ having $m$ letters equal to  $2$ and not beginning by $0$.

  The number is $\frac{2(m+1)(m+2)^2(m+1)}{3}(={1,m+3\choose m})$.
    A072819
 \end{example}
\begin{example}
  The maximal number or regions in the plane that can be formed with $m$ hyperbolas
 equals the number of ternary words of length $m+3$ having $m+1$ letters equal to
 $2$ and $0$ can not appear in the initial subword of
length $3$.

 The number is   $2(m+1)^2+1(={3,m\choose m+1})$.
  A058331
\end{example}
\begin{example}
The number ${n+1,n-1\choose n}$ equals the number of Dyck paths having exactly $n$ peaks in level $1$ and $n$ peaks in level $2$ and no other peaks.

It is also  the number of ternary words of length $2n$ having $n-1$ letters equal to $2$ and $0$ can not appear in the initial subword of
length $n+1$. A176479
\end{example}
\begin{example}
The number ${n^2,n\choose n}$ equals the number of integer points in an
$n$-dimensional sphere of Lee-radius $n^2$ centered at the origin.
It is also the number of ternary words of length $n^2+n$  having $n$ twos and $0$ can not appear in the initial subword of length $n^2$.
  A181675.
\end{example}

\end{document}